\documentclass[12pt]{amsart}
\usepackage{amssymb}
\usepackage{amsmath}
\usepackage{amsfonts}
\usepackage[all]{xy}
\usepackage{amssymb}
\usepackage{amsfonts}
\usepackage{amsmath}
\usepackage{amsthm}
\usepackage{color}
\usepackage{dsfont}
\usepackage{graphicx}
\usepackage{indentfirst}
\usepackage{latexsym}
\usepackage{mathrsfs}
\usepackage{pstricks}
\usepackage{pgf}
\usepackage{xcolor,rotating,epic,eepic}

\newtheorem{thm}{Theorem}[section]
\newtheorem{cor}[thm]{Corollary}
\newtheorem{lem}[thm]{Lemma}
\newtheorem{prop}[thm]{Proposition}

\newtheorem{dfn}[thm]{Definition}

\newtheorem{rem}[thm]{Remark}

\newcommand{\cal}{\mathcal}

\begin{document}
	
\title[Square root problem and Subnormal Aluthge transforms]{Square root problem and Subnormal Aluthge transforms}
\author[H. El Azhar, A. Hanine, K. Idrissi and E. H. Zerouali]{Hamza El Azhar $^{1}$, Abdelouahab Hanine$^{1}$, Kaissar Idrissi \MakeLowercase{and} El Hassan Zerouali $^{1}$}
\address{$^{1}$Center of mathematical research of Rabat, P.O. box 1014, Department of Mathematics, Faculty of sciences, Mohammed V university in Rabat, Rabat, Morocco.}
\email{\textcolor[rgb]{0.00,0.00,0.84}{hamza.elazhar@um5s.net.ma; a.hanine@um5r.ac.ma; i.kaissar@um5r.ac.ma; elhassan.zerouali@fsr.um5.ac.ma}}

\keywords{Finitely atomic measures, support of multiplicative convolution, subnormal Aluthge transform of weighted shifts, the	Square Root Problem for measure.}
\begin{abstract}
	For a non negative measure $\mu$ with $p$ atoms,  we study the  relation between the Square Root Problem of $\mu$ and  the problem of subnormality of ${\tilde W_\mu}$ the Aluthge transform of the associated unilateral weighted shift. We
	use an approach based on   uniquely represented elements in the support of $\mu*\mu$. We first show that if ${\tilde W_\mu}$ is subnormal, then $2p-1\le card(supp(\mu*\mu))\le [\frac{(p-1)^2+6}{2}]$. We rewrite several results known for finitely atomic measure having  at most five atoms and  give a
	complete solution for measures  six atoms.
\end{abstract}
\maketitle
\footnotetext{The  authors are supported by the CeReMaR and the  Hassan II Academy of sciences. The last author is supported by African university of Sciences and technology-Abuja. Nigeria.}
\section{Introduction}
Let us denote  $\mathcal{H}$ an infinite dimensional  Hilbert space and   let $\mathcal{L(H)}$ be  the space of all bounded linear operators on $\mathcal{H}$. An operator  $T \in \mathcal{L(H)}$ is normal if $TT^*=T^*T$,  is subnormal if it is the restriction of some normal operators and is hyponormal if $T^*T-TT^*\ge 0$. Here $T^*$ stands for the usual adjoint operator of $T$. The polar decomposition of an operator is given by the unique representation $T= U|T|$, where $|T| = (T^*T)^{\frac{1}{2}}$ and $U$ is a partial isometry satisfying $ker U = ker T$ and $ker U^*=ker T^*$. The Aluthge transform is then given by the expression 
$$\tilde T= |T|^{\frac{1}{2}}U|T |^{\frac{1}{2}}.$$
The Aluthge  transform was introduced in  \cite{al} by Aluthge, in
order to extend several  inequalities valid for hyponormal operators, and has received deep attention in the recent years.

We consider below $\mathcal{H}=l^2(\mathbb{Z}_+)$ endowed by some  orthonormal basis $\{e_n\}_{n\in\mathbb{Z}_+}$. The forward  shift operator  $W_{\alpha}$  is defined on the basis  by $W_\alpha e_n=\alpha_ne_{n+1}$, where   $\alpha=\{\alpha_n\}_{n\geq 0}$ is a given  a  sequence of positive real numbers (called {\it weights}). We associate with $W_\alpha$ the moments sequence  obtained  by $$\gamma_0=1  \mbox{  and } \gamma_k\equiv\gamma_k(\alpha):=\alpha_0^2\alpha_1^2\cdots\alpha_{k-1}^2 \mbox{ for } k\geq 1.$$ 
We will say that a sequence $\gamma$ admits a representing signed measure (called also a charge) supported in $K \subset {\mathbb R}$,  if 
\begin{equation}\label{cmp}
	\gamma_n = \int_Kt^nd\mu(t)\: \: \mbox{ for every } \: n\ge 0 \:\: \: \mbox{and} \: \: supp(\mu) \subset K.
\end{equation}
The weighted shift operator is bounded if and only if  $\|W_{\alpha}\|=sup_{n\geq0}\alpha_n<+\infty$.  It is clear  that $W_{\alpha}$ is never  normal  and that  $W_\alpha$ is hyponormal precisely when $\alpha_n$ is non decreasing. On the other hand,   Berger's Theorem says, $W_{\alpha}$  is a subnormal operator, if and only if, there exists a nonnegative Borelean measure $\mu$ (called {\it Berger measure}), which is a representing measure of $\{\gamma_n\}_{n\geq0}$ and such that $supp(\mu)\subset[0,\|W_{\alpha}\|^2]$.  In this case we will also write $W_{\alpha}=W_{\mu}$

The Aluthge transform $\tilde W_\alpha$ of a weighted shift  $W_\alpha$  is also a weighted shift, denoted below   $W_{\tilde \alpha}$ . Indeed, it is easy  to check that  $|W_\alpha| e_n=\alpha_ne_n$ and that $Ue_n=e_{n+1}$. It   follows then that
$$\tilde W_\alpha e_n= \sqrt{\alpha_n\alpha_{n+1}}e_{n+1}= {\tilde \alpha_n}e_{n+1}= W_{\tilde \alpha}e_n.$$
Notice also that ${\tilde \gamma}^2_n= \frac{1}{\alpha_0}\gamma_n \gamma_{n+1}$. In several recent  papers, the problem of subnormatiy of the Aluthge transform of  weighted shifts  was considered. See \cite{exn}, for example.
The next question has been considered in several recent papers

{\bf (SP) } Under what conditions is subnormality  of a weighted shift preserved under the Aluthge
transform?

For given two measures, let   $*$ denote the multiplicative convolution defined as follows 
$$
[\nu*\mu](A)=\int_{\mathbb{R}} \chi_A(xy) d\nu(x)d\mu(y)
$$
The square root problem is usually stated as follows

{\bf (SRP)}: Given a non negative measure $\mu$. Under what conditions, there exist a nonnegative measure $\nu$ such that $\nu*\nu= \mu$.

In \cite{Curto2019},  it has been observed that 
\begin{prop}\label{lemmeconvolution}
	Let $W_\mu$ be subnormal  weighted shift with  associated representing moment measure  $\mu$. Then  ${\tilde W}_\mu$  is subnormal  if and only if  there exists a $\mathbb{R}^+$-supported probability measure $\nu$ such that $\nu*\nu=\mu*t\mu$.  
\end{prop}
It follows from the previous proposition that 
$$ \mu \mbox{ has a square root } \Rightarrow {\tilde W}_\mu \mbox{ is subnormal} $$

The question if the reverses implication holds is stated in some recent papers. More precisely, the next question is    suggested by \cite[Question 4.1]{exn}  and stated  in \cite{Curto2019}.\\

{\bf Problem 1. }  If $W_\mu$ is  a subnormal weighted shift with Berger measure $\mu$,  are the following statements equivalent?

(i) $\mu$ as a square root;

(ii) The Aluthge transform ${\tilde W}_\mu$ is subnormal.

The case of finite atomic measures with at most 5 atoms has been completely solved in \cite{Curto2019}. The general case of finite atomic measure remains as an  open conjecture in \cite[ Conjecture 4.6]{Curto2019}.

{\bf  Conjecture 4.6.}  If $W_\mu$   subnormal with a finitely atomic Berger measure $\mu$ then the following statements are equivalent:

(i) $\mu$ has a square root;

(ii) $ {\tilde W}_\mu$ is subnormal.

The main goal of this note is to investigate Conjecture 4.6. We focus our approach on uniquely represented atoms in $supp(\mu*t\mu)$. This allows to rewrite several  proofs from \cite{Curto2019} in a short manner. We also provide an affirmative answer in the case of finite atomic Berger measures with six atoms.

\section{ Support of finite atomic measures associated with weighted shifts having  subnormal Aluthge transform.}
Before starting our study let us mention some useful remarks
\begin{enumerate}
	\item First, for  $\mu$ a non negative measure such that $\mu= a\delta_0+\mu_1$ for some $a>0$ and $0\notin supp(\mu_1)$, here $\delta_0$ is the Dirac measure at zero, we have 
	\begin{itemize}
		\item $\mu*\mu=(a^2+2a\|\mu_1\|)\delta_0 +\mu_1*\mu_1$, where $\|\mu_1\|$ is the total variation of $\mu_1$.
		\item  $\mu*t\mu=  (a\delta_0+\mu_1)*(t\mu_1)=  a\|t\mu_1\|\delta_0+\mu_1*t\mu_1. $
	\end{itemize}
	It follows that $\mu$,  (resp. $\mu*t\mu$) admits a square root if and only if $\mu_1$, (resp. $\mu_1*t\mu_1$) admits a square root. Thus in  Conjecture 4.6, we can assume without any loss of generality that $0\notin supp(\mu)$.
	\item Second, let $\mu$ be a given non negative measure and $x$ be   non negative number. Denote $\mu x(t)= \mu(xt)$ and $\mu_x(t)=\mu(t^x)$, the image measures by the mappings $t\to xt$ and $t\to t^x$,respectively. We  have the following, 
	\begin{itemize}
		\item $\mu x$ has a square root, if and only if $\mu$ has a square root. Thus, without any loss of generality, we can suppose that $\lambda_1 =1$.
		\item  It is not difficult to see that  $\mu_x$ has a square root, if and only if $\mu$ has a square root. Thus, if we put $r=\frac{\lambda_2}{\lambda_1}$, we get a root exists for one $r$ if and only if it  exists for all $ r$. This fact has been noticed in  remark 2.8 in \cite{curt1}, for $ \mu:=a_1\delta_{\lambda_1}+\cdots + a_p\delta_{\lambda_p}$ when $\lambda_i=r^i$ is a geometric sequence. 
	\end{itemize}
\end{enumerate}

In the sequel   $
\mu:=a_1\delta_{\lambda_1}+\cdots + a_p\delta_{\lambda_p}
$ is  a  $p$ atomic measures associated with a weighted shifts   $W_\mu$ and  write further 
$0< \lambda_1=1 <\lambda_2 <\cdots <\lambda_p.$
It is clear that
$$supp(\mu*t\mu) = \{\lambda_1^2 <\lambda_1\lambda_2<\lambda_2^2\le \cdots \le \lambda_{p-1}^2< \lambda_{p-1}\lambda_p<\lambda_p^2\}$$ and if $\nu$ solves  $\nu*\nu=\mu*t\mu$, we get $supp(\nu)=supp(\mu)$. In particular  $\mu:=b_1\delta_{\lambda_1}+\cdots + b_p\delta_{\lambda_p}$ for some positive numbers $b_1, \cdots, b_p$.
We also 

\begin{equation}\label{nu*nu}
	\nu*\nu= \sum_{i=1}^pb_i^2\delta_{\lambda_i^2} + 2\sum_{i\le j}^pb_ib_j\delta_{\lambda_i\lambda_j} \mbox{and } \mu*t\mu =  \sum_{i=1}^pa_i^2\lambda_i\delta_{\lambda_i^2} + \sum_{i\le j}^pa_ia_j(\lambda_i+\lambda_j)\delta_{\lambda_i\lambda_j}.
\end{equation}
The general picture of $\nu*\nu$ is given by the next diagram 
$$
\begin{array}{ccccccc}     
	\lambda_1^2&< \lambda_1\lambda_2& < \lambda_1\lambda_3&<\cdots&<\cdots&<\cdots&< \lambda_1\lambda_p\\
	& \wedge   &\wedge  & \wedge & \cdots  &\cdots  & \wedge \\
	& \lambda_2^2&<\lambda_2\lambda_3&<\cdots&<\cdots& <\cdots&<\lambda_2\lambda_p\\
	&   &\wedge  & \wedge & \cdots  &\cdots  & \wedge \\
	& & \lambda_3^2& \ddots& \vdots& \vdots& \lambda_3\lambda_p\\
	& & & \ddots& \vdots& \vdots& \vdots\\
	& & & & &  \lambda_{p-1}^2&<\lambda_{p-1}\lambda_p\\         
	& & & & &  &\wedge \\
	& & & & & &<\lambda^2_p   
\end{array}
$$

In particular $2p-1 \le card(supp(\mu*\mu))\le \frac{p(p+1)}{2}$. In the case  $\lambda_1 = 0$ we have $2p-2 \le card(supp(\mu*\mu))\le \frac{p(p-1)}{2}+1$. We will show below that the inequality $card(supp(\mu*t\mu))\le \frac{p(p+1)}{2}$ can be improved when $\mu*t\mu$ admits a square root.
\subsection{Uniquely represented atoms and subnormality of Aluthge transform.}
\begin{dfn}
	Let $\mu:=a_1\delta_{\lambda_1}+\cdots + a_p\delta_{\lambda_p}$ be a finitely atomic measure. An atom $\lambda_i\lambda_j$ in $supp(\mu *\mu)$  will be said to be uniquely represented if, for every $k,l $ satisfying 
	$\lambda_i\lambda_j= \lambda_k\lambda_l$, we have $\{i,j\}=\{k,l\}$. The subset of uniquely represented atoms will be denoted ${\cal U}r_\mu$ and the subset of non uniquely represented atoms will be denoted ${\cal NU}r_\mu=supp(\mu*\mu)\setminus {\cal U}r_\mu$.
\end{dfn}
Clearly ${\cal U}r_\mu\ne \emptyset$ since  $\{\lambda_1^2, \lambda_1\lambda_2, \lambda_{p-1}\lambda_p,\lambda_p^2\}\subset {\cal U}r_\mu$.
We assemble in the next proposition some additional  properties of ${\cal U}r_\mu$ to be used in our context.
\begin{prop}\label{regles} Suppose the exists a non negative measure $\nu$ such that  $\nu*\nu=\mu*t\mu$. Then  
	\begin{enumerate}
		\item  For every set $\{i,j, k,l \}$, containing at least 2 numbers,  we have $$  \{\lambda_i\lambda_j, \lambda_j\lambda_k, \lambda_k\lambda_l,\lambda_l\lambda_i \}  \cap{\cal NU}r_\mu\ne \emptyset.$$
		\item  For every $i\ne j$ and $i\ne k $, we have $  \{\lambda_i^2, \lambda_i\lambda_j, \lambda_j\lambda_k, \lambda_i\lambda_k\} \cap{\cal NU}r_\mu \ne \emptyset.$
		\item For every $i\ne j$, we have $  \{\lambda_i^2,\lambda_i\lambda_j,\lambda_j^2\} \cap{\cal NU}r_\mu\ne \emptyset.$
		\item  If $ \{\lambda_i^2, \lambda_i\lambda_j, \lambda_j\lambda_k, \lambda_k^2\} \subset {\cal U}r_\mu$ for some  $i\ne j\ne k $, then $  \lambda_j^2=\lambda_i\lambda_k.$
		\item  $\{ \lambda_2^2, \lambda_{p-1}^2, \lambda_1\lambda_p, \lambda_1\lambda_{p-1}, \lambda_2\lambda_p\} \subset {\cal NU}r_\mu.$
	\end{enumerate}
\end{prop}

\begin{proof}
	We have $$\nu*\nu= \sum_{i=1}^pb_i^2\delta_{\lambda_i^2} + 2\sum_{i\le j}^pb_ib_j\delta_{\lambda_i\lambda_j} \mbox{ and }  \mu*t\mu =  \sum_{i=1}^pa_i^2\lambda_i\delta_{\lambda_i^2} + \sum_{i\le j}^pa_ia_j(\lambda_i+\lambda_j)\delta_{\lambda_i\lambda_j}.$$ 
	Then by matching the coefficients  of the adequate  masses in both sides, we derive that  for  $\lambda_i\lambda_j \in  {\cal U}r_\mu,  $ we have  $ 2b_ib_j= a_ia_j(\lambda_i+\lambda_j),$ ($ b_i^2=a_i^2\lambda_i,$ if $i=j$). 
	\begin{enumerate}
		\item Suppose $i\ne k$ and $j\ne l$ are such that $  \{ \lambda_i\lambda_j, \lambda_j\lambda_k, \lambda_k\lambda_l,\lambda_l\lambda_i, \}  \subset {\cal U}r_\mu $. We get 
		$$ \left\{\begin{array}{rcl}  2b_ib_j=a_ia_j(\lambda_i+\lambda_j), & & 2b_jb_k=a_ja_k(\lambda_j+\lambda_k)\\
			2b_kb_l=a_ka_l(\lambda_k+\lambda_l), & & 2b_lb_i=a_la_i(\lambda_l+\lambda_i).
		\end{array}
		\right.
		$$
		Multiplying both  sides gives $(\lambda_i+\lambda_j)(\lambda_k+\lambda_l)  =  (\lambda_j+\lambda_k)(\lambda_l+\lambda_i)$ and expanding, we obtain
		$$(\lambda_i-\lambda_k)(\lambda_l-\lambda_j) =0. $$ Contradiction.
		\item 
		Derives from 1. with $i= l$.
		\item  Derives from 2. with $j= k$.
		\item If $ \{\lambda_i^2, \lambda_i\lambda_j, \lambda_j\lambda_k, \lambda_k^2\} \subset {\cal U}r_\mu$, we will derive that 
		$$ 
		\left\{\begin{array}{lcl}
			b_i= a_i\sqrt{\lambda_i}, & &
			2b_ib_j= a_ia_j(\lambda_i+\lambda_j)\\
			2b_kb_j= a_ka_j(\lambda_k+\lambda_j),& & b_k=  a_k\sqrt{\lambda_k}, \\  
		\end{array}
		\right.
		$$
		which gives  by multiplying both sides 
		$$
		(\lambda_j+\lambda_i)\sqrt{\lambda_k} = (\lambda_j+\lambda_k)\sqrt{\lambda_i}.
		$$
		And then  $$(\sqrt{\lambda_k} -\sqrt{\lambda_i})(\lambda_j-\sqrt{\lambda_i\lambda_k})=0.$$
		Now, since $\lambda_i\ne \lambda_k$; we get the desired result.
		\item $\{ \lambda_2^2, \lambda_{p-1}^2, \lambda_1\lambda_p, \lambda_1\lambda_{p-1}, \lambda_2\lambda_p\} \cap {\cal U}r_\mu =\emptyset.$ The case $p=3$ is trivial, we  suppose that $p\ge 4$.
		\begin{enumerate}
			\item For $\{ \lambda_2^2, \lambda_{p-1}^2, \lambda_1\lambda_p\} \cap {\cal U}r_\mu =\emptyset$, we use $4.$ with $(i,j)=(1,2)$, $(i,j)=(p-1,p)$ and   $(i,j)=(1,p)$, respectively.
			\item  The cases $ \lambda_2\lambda_{p}$ and   $\lambda_1\lambda_{p-1} $   are symmetric. We only show $ \lambda_2\lambda_{p}\notin {\cal U}r_\mu$.
			If $ \lambda_2\lambda_{p}\in {\cal U}r_\mu, $ we will get  $\{\lambda_1\lambda_2,  \lambda_2\lambda_p\} \subset {\cal U}r_\mu $,  and hence 
			$ \lambda_2^2= \lambda_1\lambda_p$. In particular, it follows that   $\lambda_1\lambda_{p-1}\in {\cal U}r_\mu $ and as before 
			$ \lambda_{p-1}^2= \lambda_1\lambda_p$. Contradiction.
		\end{enumerate}
	\end{enumerate}
\end{proof}

We deduce  the next immediate corollaries
\begin{cor}\label{lambda1lambdap}
	For every $1< k<p$, $\{\lambda_1\lambda_k,\lambda_2\lambda_k\}\cap{\cal NU}r_\mu \ne \emptyset $ and $\{ \lambda_{p-1}\lambda_k,\lambda_p\lambda_k \}\cap{\cal NU}r_\mu \ne \emptyset.$ Furthermore, there is at most one $k$ such that $\{\lambda_1\lambda_k,\lambda_p\lambda_k \}\subset {\cal U}r_\mu$ and  if such $k$ exists, we will have $\lambda_k^2=\lambda_1\lambda_p$.
\end{cor}
\begin{proof} We use 1. with  $\{\lambda_1^2,  \lambda_1\lambda_2, \lambda_1\lambda_k,\lambda_2\lambda_k\}\cap{\cal NU}r_\mu \ne \emptyset $ and $\{ \lambda_{p-1}\lambda_k,\lambda_p\lambda_k, \lambda_{p-1}\lambda_p, \lambda_p^2 \}\cap{\cal NU}r_\mu\ne \emptyset$ respectively. For the additional assumption, we use 4. in Proposition \ref{regles}.
\end{proof}
\begin{cor}\label{lambda1l}
	For every $2< k<p-1$, any of the sets $\{\lambda_1\lambda_k,\lambda_2\lambda_k, \lambda_k^2\}\cap{\cal NU}r_\mu $ and $\{\lambda_p\lambda_k,\lambda_{p-1}\lambda_k, \lambda_k^2\}\cap{\cal NU}r_\mu $ contains at least two elements.
\end{cor}
\begin{proof}  Derives from  $\{\lambda_1\lambda_k, \lambda_2\lambda_k\}\cap{\cal NU}r_\mu \ne \emptyset, $ $\{\lambda_1\lambda_k, \lambda_k^2\}\cap{\cal NU}r_\mu\ne \emptyset $ and $\{\lambda_2\lambda_k, \lambda_k^2\}\cap{\cal NU}r_\mu \ne \emptyset.$
\end{proof}
\subsection{Number of atoms of $p$ atomic measure such that $ {\tilde W}_\mu$ is subnormal}
We start with the next auxiliary immediate lemma of independent interest. 
\begin{lem}
	Let $\mu$ be a non negative measure and let ${\cal NU}r_\mu$ the set of non uniquely represented elements in $\mu*t\mu$. Then
	$$card(supp(\mu*t\mu)) \le \left[ \frac{p(p+1)-card({\cal NU}r_\mu)}{2}\right]$$
\end{lem}
\begin{figure}
	$$
	\begin{array}{cc|ccccc|cc}     
		{\lambda_1^2}&< \lambda_1\lambda_2& < \lambda_1\lambda_3&<\cdots&<\fbox{$\lambda_1\lambda_k$}&<\cdots&<\lambda_1\lambda_{p-2}&<\bf{\lambda_1\lambda_{p-1}}&< \bf{ \lambda_1\lambda_p}\\
		& \wedge   &\wedge  & \wedge & \wedge &\wedge &\wedge &\wedge & \wedge \\
		& \bf{\lambda_2^2}&<\lambda_2\lambda_3&<\cdots&<\fbox{$\lambda_2\lambda_k$}&<\cdots &<\lambda_2\lambda_{p-2}&<\lambda_2\lambda_{p-1}&<\bf{\lambda_2\lambda_p}\\\hline
		&   &\wedge  & \wedge &\cdots& \cdots  &\cdots & \cdots  & \wedge \\ 
		& & \lambda_3^2& \ddots& \ddots&\cdots& &\vdots <\lambda_3\lambda_{p-1}&< \lambda_3\lambda_p\\
		&   &  &  &\cdots& \cdots &\vdots &\cdots  & \wedge \\ 
		& & & \ddots& \vdots& &\vdots &\vdots& \vdots\\        
		&   &  & &\wedge& &  &\wedge  & \wedge \\ 
		& & && \fbox{$ \lambda_{k}^2$} &\dots&\vdots & <\fbox{$\lambda_{k}\lambda_{p-1}$}& <\fbox{$\lambda_{k}\lambda_p$}\\
		& & & &&\ddots &\vdots &\vdots& \vdots\\        
		& & &&& & \lambda_{p-2}^2& \lambda_{p-2}\lambda_{p-1}& \lambda_{p-2}\lambda_p\\
		\hline
		& & & &&& & \bf{\lambda_{p-1}^2}&<\lambda_{p-1}\lambda_p\\         
		& & & && &&  &\wedge \\
		& & & && && &<\lambda^2_p   
	\end{array}
	$$
	\caption{ Bold ${\cal NU}r_\mu$  elements and $3(p-4)$ ${\cal NU}r_\mu$ elements from $\{\lambda_1\lambda_k,\lambda_2\lambda_k,\lambda_k^2,\lambda_k\lambda_{p-1},\lambda_k\lambda_p\}\; \; 3\le k\le p-2.$}
	\label{fig1}
\end{figure}
Noticing that if $\mu*t\mu$ admits a square root, then for $p\ge 4$, we have ${\cal NU}r_\mu$  contains at least $5$ elements given  by Proposition \ref{regles}  (in bold in Figure \ref{fig1}), $2(p-4)$ elements coming from the $p-4$ central columns by corollary \ref{lambda1l}, and  additional $p-4$ elements from the two last columns by the second part of the corollary \ref{lambda1lambdap} \, ( 
boxed in Figure \ref{fig1}), we derive that $ card({\cal NU}r_\mu)\ge 3p-7.$ From the previous discussion and Proposition \ref{lemmeconvolution}, we deduce 
\begin{prop}\label{card} Let $p\ge 4$. If $\mu$ is a $p$ atomic measure such that $ {\tilde W}_\mu$ is subnormal, then $$2p-1\le card(supp(\mu*t\mu)) \le  \left[\frac{(p-1)^2+6}{2}\right].$$
\end{prop}

\begin{rem}
	The two examples below show that  both inequalities in Proposition \ref{card} are sharp, 
	\begin{enumerate}
		\item Take $\mu=\frac{1}{8}\delta_1+\frac{\sqrt{2}-1}{2}\delta_2+\frac{7-4\sqrt{2}}{4}\delta_4+\frac{\sqrt{2}-1}{2}\delta_8+\frac{1}{8}\delta_{16}$. Then, $\nu=\frac{\sqrt{2}}{4}\delta_{1}+\frac{2-\sqrt{2}}{2}\delta_{2}+\frac{\sqrt{2}}{4}\delta_{4}$ is a square root of $\mu$ and then  ${\tilde W}_\mu$ is subnormal, a direct calculation gives :
		$$
		card( supp(\mu*t\mu))=card(\left\{1, 2, 4, 8, 16, 32, 64, 128, 256\right\})=9.
		$$
		\item Take $\mu=\frac{1}{4}\delta_1+\frac{1}{3}\delta_3+\frac{1}{6}\delta_6+\frac{1}{9}\delta_9+\frac{1}{9}\delta_{18}+\frac{1}{36}\delta_{36}$. Then, $\nu=\frac{1}{2}\delta_{1}+\frac{1}{3}\delta_{3}+\frac{1}{6}\delta_{6}$ is a square root of $\mu$ and then  ${\tilde W}_\mu$ is subnormal, a direct calculation gives :
		$$
		card(supp(\mu*t\mu))=card(\left\{1, 3, 6, 9, 18, 27, 36, 54, 81, 108, 162, 216, 324, 648, 1296 \right\})=15.
		$$
	\end{enumerate}
\end{rem}

Since $\{\lambda_2^2,\lambda_{p-1}^2\}\subset  {\cal NU}r_\mu,$ there exists $i\le p-2$ and $j \ne 3$ such that $\lambda_2^2= \lambda_1\lambda_j$ and 
$\lambda_{p-1}^2= \lambda_i\lambda_p$ for some $i\le p-2$. In particular, if $p=3$, we get necessarily  $\lambda_2^2=  \lambda_1\lambda_p=\lambda_1\lambda_3.$ The reverse result is given by the next corollary 

\begin{cor}\label{lambda2lambdap-1}
	Under the notations above with $p\ge 3$. If the identities $\lambda_2^2= \lambda_1\lambda_p$ or   $\lambda_{p-1}^2= \lambda_1\lambda_p$, then $p=3$.
\end{cor}
\begin{proof}
	If we assume that $\lambda_2^2=\lambda_1\lambda_p$, then $\{\lambda_1\lambda_{p-1}, \lambda_{p-1}\lambda_p\}\subset{\cal U}r_\mu$. By Corollary \ref{lambda1lambdap}, we get $\lambda_{p-1}^2=\lambda_1\lambda_p=\lambda_2^2$, and thus $p=3$.
\end{proof}
For $p\ge 4$, the above result is improved as follows.
\begin{cor}\label{lem2.5} Under the notations above, if $ (\lambda_2^2,\lambda_{p-1}^2)= (\lambda_1\lambda_{p-1}, \lambda_2\lambda_{p})$ then $p = 4$. 
\end{cor}
\begin{proof}
	Let $p\ge 6$ and  assume  $\lambda_2^2= \lambda_1\lambda_{p-1}$ and  $\lambda_{p-1}^2= \lambda_2\lambda_{p}$.   From the inequalities 
	$$\lambda_1^2 < \lambda_1\lambda_2 < \lambda_1\lambda_3<\cdots \lambda_1\lambda_{p-1}=  \lambda_2^2 \, \mbox{  and } \, 
	\lambda_{p-1}^2 =\lambda_2\lambda_p < \lambda_3\lambda_p<\cdots <  \lambda_p^2,  $$
	it comes that $\lambda_1\lambda_2 < \lambda_1\lambda_3<\cdots \lambda_1\lambda_{p-2}$ 
	and $ \lambda_3\lambda_p < \lambda_4\lambda_p<\cdots <  \lambda_p^2 $
	are   uniquely represented. In particular, $\{\lambda_1^2 , \lambda_1\lambda_3 , \lambda_3\lambda_p ,\lambda_p^2 \} \subset {\cal U}r_\mu$ and  $\{\lambda_1^2 , \lambda_1\lambda_4, \lambda_4\lambda_p, \lambda_p^2  \} \subset {\cal U}r_\mu$. We use Proposition \ref{regles} again to get  $\lambda_3^2 =\lambda_1\lambda_p= \lambda_4^2$. Thus $\lambda_3=\lambda_4.$  Contradiction.\\
	For $p=5$, suppose $\lambda_2^2= \lambda_1\lambda_{4}$ and  $\lambda_{4}^2= \lambda_2\lambda_{5}$. We derive in particular that $\{\lambda_1\lambda_3 ,  \lambda_3\lambda_5\} \subset  {\cal NU}r_\mu $ and then $\lambda_3^2=\lambda_1\lambda_5. $ This last fact forces $\{\lambda_1\lambda_3 ,  \lambda_2\lambda_3\} \subset  {\cal U}r_\mu $ and contradicts Proposition \ref{regles}.
\end{proof}
\section{Finite atomic measures associated with weighted shifts having  subnormal Aluthge transform.}
In the sequel, we will say that  $p$ atomic measure $\mu$ has a geometric support if it is of the next form
$$ supp(\mu ) = \{a,ar,\cdots,ar^{p-1}\} \: \: \: \mbox{ for some } a, r  \in {\mathbb R}^*.$$
We start with the next useful lemma
\begin{lem}\label{geom} Let $\mu$ be $p-$atomic such that  $ {\tilde W}_\mu$ is subnormal. Then  $\mu$ has a geometric support if and only if  $card(supp(\mu*t\mu)) =2p-1$.
\end{lem}
\textit{Proof.} For the direct implication, suppose $supp(\mu ) = \{a,ar,\cdots,ar^{p-1}\}$. Direct computations, give
$supp(\mu*t\mu)= \{a^2, a^2r^{}, \cdots,a^2r^{2p-2}\} $ and hence  $card(supp(\mu*t\mu)) =2p-1$.

Conversely, write $supp(\mu)= \{\lambda_1,\cdots, \lambda_p\}$ and suppose that  $card(supp(\mu*t\mu)) =2p-1$. It follow that $$supp(\mu*t\mu)=\{\lambda_1^2<\lambda_1\lambda_2<\lambda_2^2<\cdots< \lambda_{p-1}^2<\lambda_{p-1}\lambda_p<\lambda_p^2\}.$$
Since   
$
\lambda_i\lambda_{i+2}\in ]\lambda_i\lambda_{i+1},\lambda_{i+1}\lambda_{i+2}[\cap supp(\mu*t\mu)=\{\lambda_{i+1}^2\}, 
$ for every $1\le i \le p-2$, 
we derive that   $\lambda_i\lambda_{i+2}=\lambda_{i+1}^2$, and then $\frac{\lambda_{i+1}}{\lambda_i}=\frac{\lambda_{i+2}}{\lambda_{i+1}}$. Thus $r=:\frac{\lambda_2}{\lambda_1}= \frac{\lambda_3}{\lambda_2}=\cdots=\frac{\lambda_{p}}{\lambda_{p-1}}$. Finally $\mu$ has a geometric support with $a=\lambda_1$ and 
$r=:\frac{\lambda_2}{\lambda_1}$.
\subsection{The case   $card(supp(\mu))= 3$.}
\begin{prop}\label{3atoms} Under the notations above, if $p=3$ then  the following are equivalent
	\begin{enumerate}
		\item $ {\tilde W}_\mu$ is subnormal;
		\item $\mu = a_1\delta_{\lambda_1}+a_2\delta_{\lambda_1r} +a_3\delta_{\lambda_1r^2}$  with $ r=\frac{\lambda_2}{\lambda_1}
		>1$ and where  $a_1, a_2, a_3$ are non negative numbers 
		satisfying  $a_1+a_2+a_3=1$ and $a_2^2=4a_1a_3$.
		\item $\mu $ admits a square root.
	\end{enumerate}
\end{prop}
\begin{proof} $(1.\Rightarrow 2.)$ Suppose that $ {\tilde W}_\mu$ is subnormal.  
	Since $supp(\mu*t\mu)=\{\lambda_1^2,\lambda_1\lambda_2, \lambda_{2}^2, \lambda_{1}\lambda_3, \lambda_{2}\lambda_3,\lambda_3^2\}$ and ${\cal NU}r_\mu = \{\lambda_2^2, \lambda_{1}\lambda_3\}$ it follows that $\lambda_{2}^2= \lambda_{1}\lambda_3$ and $card(supp(\mu))=5=2\times3-1,$ and  hence  $\mu$ has a geometric support and then  $\mu = a_1\delta_{\lambda_1}+a_2\delta_{\lambda_1r} +a_3\delta_{\lambda_1r^2}$.
	
	Now writing   $\mu*t\mu=\nu*\nu$ for some $\nu=b_1\delta_{\lambda_1}+b_2\delta_{\lambda_1r} +b_3\delta_{\lambda_1r^2}$, we get
	$$\begin{array}{ll}
		\mu*t\mu= & a_1^2\lambda_1\delta_{\lambda_1^2}+a_1a_2\lambda_1(1+r)\delta_{\lambda_1^2 r}+\lambda_1(a_1a_3(1+r^2)+a_2^2r)\delta_{\lambda_1^2 r^2} \\ & +a_2a_3\lambda_1(r+r^2)\delta_{\lambda_1^2 r^3}+   a_3^2\lambda_1 r^2\delta_{\lambda_1^2 r^4}\\ \nu *\nu = & b_1^2\delta_{\lambda_1^2}+2b_1b_2\delta_{\lambda_1^2 r}+(2b_1b_3+b_2^2)\delta_{\lambda_1^2r^2} +2b_2b_3\delta_{\lambda_1^2 r^3}+   b_3^2\delta_{\lambda_1^2 r^4}.\end{array}$$\\
	Identifying the coefficient of the  atoms, we obtain $b_1= a_1\sqrt{\lambda_1}$, $b_2 = a_2\sqrt{\lambda_1}\frac{1+r}{2}$ and  
	$b_3= a_3r\sqrt{\lambda_1}$. Using the identity $2b_1b_3+b_2^2=\lambda_1(a_1a_3(1+r^2)+a_2^2r) $ we get  $ a_2^2=4a_1a_3$.\\
	$(2.\Rightarrow 3.)$ Let  $\mu = a_1\delta_{\lambda_1}+a_2\delta_{\lambda_1r} +a_3\delta_{\lambda_1 r^2}$, where   $a_1, a_2, a_3$ are non negative numbers 
	satisfying  $a_1+a_2+a_3=1$ and $a_2^2=4a_1a_3$. Then  $\mu_0=\sqrt{a_1}\delta_{\sqrt{\lambda_1}}+\sqrt{a_3}\delta_{r\sqrt{\lambda_1}}$ satisfies $\mu = \mu_0*\mu_0$.\\ 
	$(3.\Rightarrow 1.)$ Derives from $\mu = \mu_0*\mu_0\Rightarrow \mu*t\mu= \mu_0*t\mu_0* \mu_0*t\mu_0$ and Proposition \ref{lemmeconvolution}.
\end{proof}
\subsection{The case   $card(supp(\mu))= 4$.}
\begin{prop} Under the notations above, if $p=4$ then $ {\tilde W}_\mu$ is not subnormal.\end{prop}
\begin{proof} Assume $ {\tilde W}_\mu$ is  subnormal. It will follow that $$7\le card(supp(\mu*t\mu) )\le \left[\frac{(4-1)^2+6}{2}\right] = 7,$$ and hence
	$supp(\mu)$ is geometric. Assume $\lambda_1=1$. We put $\lambda_i= r^{i-1}$ for $1\le i\le 4$ and write
	$\mu=\sum\limits_{i=0}^{i=3}a_i\delta_{r^i}$, we get\\
	$\mu*t\mu=
	a_0^2\delta_{1}+a_0a_1(1+r)\delta_{r}+(a_1^2r+a_0a_2(1+r^2))\delta_{r^2}+(a_1a_2(r+r^2)+a_0a_3(1+r^3))\delta_{r^3}+(a_2^2r^2+a_1a_3(r+r^3))\delta_{r^4} +a_2a_3(r^2+r^3)\delta_{r^5}+a_3^2r^3\delta_{r^6}.
	$ 
	
	Also, we set $\nu=\sum\limits_{i=0}^{i=3}b_i\delta_{r^i}$.
	for $\nu$ such that $\mu*t\mu=\nu*\nu$. Matching the coefficients of ${\cal U}r_\mu$ elements, we derive the identities 
	$$ b_0^2 = a_0^2, \; 
	2b_0b_1 = a_0a_1(1+r), \; 
	2b_2b_3= a_2a_3(r^2+r^3) \; 
	\mbox{ and } \;  b_3^2=r^3a_3^2.
	$$
	From which, we deduce $
	b_0 = a_0 , \;  b_1= a_1\frac{(1+r)}{2}, \;
	b_2= a_2\frac{(1+r)\sqrt{r}}{2}, \; \mbox{ and } 
	b_3=a_3r\sqrt{r}.$\\
	Plugging in the additional equation 
	$$
	2b_0b_2+b_1^2 =  a_0a_2(1+r^2)+a_1^2r, 
	$$
	we obtain
	$$0=a_0a_2(1+r)\sqrt{r}+a_1^2\frac{(1+r)^2}{4} 
	-  a_0a_2(1+r^2)-a_1^2r.$$
	The previous equality holds for arbitrary $r$, as noticed  in the beginning of  the second section,  and hence all  coefficients in the polynomial equation (in $\sqrt{r}$) vanish. In particular, with $r\sqrt{r}$,  we deduce that $a_0a_2=0$. Contradiction
\end{proof}
Since every measure with a square root satisfies $ {\tilde W}_\mu$ is  subnormal,  we deduce the next corollary, 
\begin{cor}
	Under the notations above, if $p=4$ then $\mu$ has no square root.
\end{cor}
Recal that  we are assuming here that $0\not\in supp(\mu)$. The case where $0$ is in the support the result derives from   3 atomic  support measure satisfying proposition \ref{3atoms}.
\subsection{The case   $card(supp(\mu))= 5$.}
The main result in this subsection  stated next is form \cite{Curto2019}. We give a simpler proof based on uniquely represented elements.
\begin{prop}\label{case5} Let $\mu=\sum_{i=1}^5a_i\delta_{\lambda_i}$ a probability measure with $0<\lambda_1<\cdots <\lambda_5$. The following are equivalent
	\begin{enumerate}
		\item $\mu*t\mu$ has a square  root
		\item $\mu$ has geometric support and $\frac{a_2^2}{a_1}  = \frac{a_4^2}{a_5} $ and  $a_3= \frac{a_2^2}{4a_1} +2\sqrt{a_1 a_5}.$
		\item $\mu$ has a square root.
	\end{enumerate}  
\end{prop}
\textit{Proof.} Since $\mu*t\mu$ has a square  root, we get $$9\le card(supp(\mu*t\mu) )\le \left[\frac{(5-1)^2+6}{2}\right] = 11.$$
We first prove that $card(supp(\mu*t\mu) )\ne 11$. Counting non uniquely represented elements, it suffices to show that $\lambda_2\lambda_4 \in {\mathcal NU}r_\mu $. Indeed,  if $\lambda_2^2  = \lambda_1\lambda_4 $ we get $  \lambda_2\lambda_3=\lambda_1\lambda_5 $ and then $\lambda_3\lambda_4= \lambda_2\lambda_5. $ Which forces $\lambda_3^2  = \lambda_2\lambda_4.$  If $\lambda_2^2  = \lambda_1\lambda_3 $, we distinguish two subcases,
\begin{itemize}
	\item $     \lambda_1\lambda_4= \lambda_2\lambda_3\Rightarrow \lambda_3^2=\lambda_2\lambda_4.$ 
	\item $   \lambda_1\lambda_4= \lambda_3^2.$ Then since, $ \lambda_4^2=\lambda_3\lambda_5$, we get $ \lambda_1\lambda_5= \lambda_3\lambda_4$ and hence $\lambda_2\lambda_5\in {\mathcal U}r_\mu $.  Contradiction.
\end{itemize}
We claim that $card(supp(\mu*t\mu) )=9$ and hence that   $\mu$ has a geometric support. To this aim let us first suppose that $card(supp(\mu*t\mu) )=10$ and show that $\lambda_2^2=\lambda_1\lambda_3 $ and   $\lambda_4^2  = \lambda_3\lambda_5$. 
If $\lambda_2^2\ne \lambda_1\lambda_3 $, we get $\lambda_1\lambda_3\in {\mathcal U}r_\mu$ and $\lambda_4^2  = \lambda_3\lambda_5$. Thus 
$$\begin{array}{lll} supp(\mu*t\mu)&=&\{\lambda_1^2<\lambda_1\lambda_2<\lambda_1\lambda_3 <\lambda_2^2<\lambda_2\lambda_3<\, \lambda_3^2\, <\lambda_3\lambda_4<\,  \lambda_4^2\, <\lambda_4\lambda_5< \lambda_5^2  \} \\ &=&\{\lambda_1^2<\lambda_1\lambda_2<\lambda_1\lambda_3 <\lambda_2^2< \lambda_2\lambda_3<\lambda_2\lambda_4<\lambda_2\lambda_5<\lambda_3\lambda_5<\lambda_4\lambda_5< \lambda_5^2 \}.
\end{array} $$
We deduce that $ \; \lambda_3^2=\lambda_2\lambda_4,$ 
from which we get $$\frac{\lambda_3}{\lambda_2} = \frac{\lambda_4}{\lambda_3} = \frac{\lambda_5}{\lambda_4} (=r), $$
and together with  the  additional equality $\lambda_2^2=\lambda_1\lambda_4$, we obtain $$ \lambda_2=\lambda_1r^2, \lambda_3=\lambda_1r^3, \lambda_4=\lambda_1r^4 \mbox{ and } \lambda_5=\lambda_1r^5.$$ 
We put now  $\mu=a_1\delta_{\lambda_1}+ \sum_{i=2}^5 a_i \delta_{\lambda_1 r^i}$, and $\nu=b_1\delta_{\lambda_1}+ \sum_{i=2}^5 b_i \delta_{\lambda_1 r^i}$. Matching the uniquely represented elements in $\mu*t\mu=\nu*\nu$ we get:
$  b_1^2 = a_1^2\lambda_1, \;  2b_1b_2 = a_1a_2\lambda_1(1+r^2), \; 
2b_1b_3= a_1a_3{\lambda_1}(1+r^3), \; 2b_4b_5=a_4a_5{\lambda_1}(r^4+r^5)$ and $ 
b_5^2=a_5^2{\lambda_1}r^5. $
Then
$$ b_1 = a_1\sqrt{\lambda_1}, \; b_2 = a_2\sqrt{\lambda_1}\frac{(1+r^2)}{2}, \;  b_3= a_3\sqrt{\lambda_1}\frac{(1+r^3)}{2}, \:
b_4=a_4\sqrt{\lambda_1}r^{\frac32}\frac{(1+r)}{2}, \; 
b_5=a_5\sqrt{\lambda_1}r^{\frac52}$$
Using the additional equation 
$$ b_3^2+2b_2b_4=\lambda_1a_3^2r^3+\lambda_1a_2a_4(r^2+r^4),$$ we get $$
a_3^2\frac{(1+r^3)^2}{4}+a_2a_4r^{\frac32}\frac{(1+r)(1+r^2)}{2}-a_3^2r^3+a_2a_4(r^2+r^4)=0$$
and then $a_3= 0$ as the coefficient in $r^6$. Contradiction.

So $\lambda_2^2=\lambda_1\lambda_3$, and similarly $\lambda_4^2=\lambda_3\lambda_5$. We have  $\lambda_1\lambda_5 \in {\cal NU}r_\mu$ and  $\lambda_1\lambda_5 \in \{\lambda_3^2, \lambda_2\lambda_4\}$, (otherwise, either $\lambda_1\lambda_4$, or $\lambda_2\lambda_5$ is uniquely represented, which is not true). Now in both cases, we get $\mu$ has geometric support.

It remains to characterise measures $\mu$ with   geometric support such that  $\mu*t\mu$ has a square  root. We have the following
\begin{lem}\label{lem36} Under the notations above, suppose $\lambda_1=1$ and let $\mu=\sum_{i=1}^5a_i \delta_{r^i}$. Then $\mu*t\mu$ has a square root if and only if  $\frac{a_2^2}{a_1}  = \frac{a_4^2}{a_5} $ and  $a_3= \frac{a_2^2}{4a_1} +2\sqrt{a_1 a_5}.$
\end{lem}

To this aim, suppose $\lambda_i= r^{i-1}$ for $1\le i\le 5$. The identity $ \nu*\nu=\mu*t\mu$ provides 9 non linear equations. We  use  the  4  equations deriving from uniquely represented elements,
$$ b_1^2 = a_1^2,  \;  \; 
2b_1b_2 = a_1a_2(1+r),  \; \;  
2b_4b_5 = a_4a_5(r^3+r^4), \; \; 
\mbox{ and } b_5^2=a_5^2r^4,   
$$ 
to get     $$  b_1 = a_1, \; 
b_2 = a_2\frac{(1+r)}{2}, \; 
b_4=a_4\frac{r+r^2}{2} \; \text{ and } \; 
b_5=a_5r^2.$$

Comparing the expression of  $b_3$ obtained from  the next two  equations.  
$$
\left\{ 
\begin{array}{lll}
	2b_1b_3+ b_2^2 &=& a_1a_3(1+r^2)+a_2^2r\\
	2b_3b_5+ b_4^2 &=& a_3a_5(r^2+r^4)+a_4^2r^3\\
\end{array}
\right.
$$
we derive $$b_3=a_3\frac{(1+r^2)}{2}-\frac{a_2^2}{a_1}\frac{(1-r)^2}{8} =a_3\frac{(1+r^2)}{2}-\frac{a_4^2}{a_5}\frac{(1-r)^2}{8} .$$ Thus 
\begin{equation}
	a_2^2a_5=a_4^2a_1.
\end{equation}
Now,
Because of the additional compatibility equation
$$ 2b_1b_5+ b_3^2+ 2b_2b_4 = a_1a_5(1+r^4)+a_3^2r^2+ a_2a_4(r+r^3)$$  we obtain
$$a_3= \frac{a_2^2}{4a_1} +2\sqrt{a_1 a_5}.$$

To end the proof of Lemma \ref{lem36} and Proposition \ref{case5}, we 
use , as in  \cite[Theorem 4.2]{Curto2019}, the measure $\nu =\sqrt{a_1}\delta_{r}+\frac{a_2}{2\sqrt{a_1}}\delta_{r^3}+\sqrt{a_5}\delta_{r^5}$  to provide a square root of $\mu$.        
\section{The case of Atomic measures with  $6$ atoms.}
In this case, when $ {\tilde W}_\mu$ is  subnormal, we get $$11\le card(supp(\mu*t\mu) )\le \left[\frac{(6-1)^2+6}{2}\right] = 15.$$ 
We will see in this section that the above inequality is optimal.

We start with the next lemma which improves Lemma \ref{lambda2lambdap-1}.
\begin{lem}\label{lem2.6}Let $\mu=\sum_{i=1}^6a_i\delta_{\lambda_i}$  as above such that 
	$\mu*t\mu$ has a square  root. Then $\lambda_2^2\ne \lambda_1\lambda_{5}$ and  $\lambda_{5}^2\ne  \lambda_2\lambda_{6}$.
\end{lem}
\begin{proof}

	Because of symmetry, we only show that $\lambda_2^2\ne \lambda_1\lambda_5$. If it is not the case, we will get $r_1=\frac{\lambda_2}{\lambda_1} = \frac{\lambda_5}{\lambda_2} $ and 
	$\{ \lambda_1^2,\lambda_1\lambda_2,\lambda_1\lambda_3,\lambda_1\lambda_4\}\subset{\cal U}r_\mu$, so by  Proposition \ref{regles} $\{ \lambda_2\lambda_3,\lambda_2\lambda_4\}\subset {\cal NU}r_{\mu} $ 
	
	We will have  necessarily, 
	\begin{equation}\label{(2,3)=(1,6)}
		\lambda_2\lambda_3=\lambda_1\lambda_6 \mbox{ and } \lambda_3^2=\lambda_2\lambda_4.
	\end{equation}
	Then   $\frac{\lambda_2}{\lambda_1} =\frac{\lambda_6}{\lambda_3} = \frac{\lambda_5}{\lambda_2}  $ and $\frac{\lambda_3}{\lambda_2}=\frac{\lambda_4}{\lambda_3}$. From Proposition \ref{regles}  again $\{\lambda_3\lambda_6, \lambda_4\lambda_6, \lambda_5^2\} \subset {\cal NU}r_{\mu}. $ Thus 
	$\lambda_5^2=\lambda_4\lambda_6$ (or equivalently $\frac{\lambda_5}{\lambda_4}=\frac{\lambda_6}{\lambda_5}$). Otherwise $\lambda_5^2=\lambda_3\lambda_6 $ and then $\lambda_4\lambda_6\in {\cal U}r_{\mu}$. Contradiction.
	
	Writing    $\lambda_2\lambda_4\lambda_5=\lambda_3^2\lambda_5=\lambda_2\lambda_6\lambda_3$, we derive that 
	$ \lambda_5\lambda_4=\lambda_3\lambda_6,  $
	which leads to $\frac{\lambda_3}{\lambda_2}=\frac{\lambda_4}{\lambda_3}=\frac{\lambda_5}{\lambda_4}=\frac{\lambda_6}{\lambda_5}=r_2.$ We deduce that  $r_1=\frac{\lambda_6}{\lambda_3} =r_2^3$ and then 
	$\lambda_k= r^{k+1}\lambda_1$ for $1 <k\le 6$, with   $r=r_2$.  
	Since $\{ \lambda_1^2,\lambda_1\lambda_2,\lambda_1\lambda_3,\lambda_1\lambda_4\}\subset{\cal U}r_\mu$, we get:
	$$
	b_1^2=a_1^2\lambda_1,\; \; 
	2b_1b_2=a_1a_2\lambda_1(1+r^3), \; \; 
	2b_1b_3=a_1a_3\lambda_1(1+r^4), \;  \mbox{  and  }        
	2b_1b_4=a_1a_4\lambda_1(1+r^5). $$
	That gives $b_1=a_1\sqrt{\lambda_1}, \; b_2=a_2\sqrt{\lambda_1}\frac{1+r^3}{2}, \; b_3=a_3\sqrt{\lambda_1}\frac{1+r^4}{2}$ and $\; b_4=a_4\sqrt{\lambda_1}\frac{1+r^5}{2}.$\\
	Replacing in the additional equation $\, \, b_3^2+2b_2b_4=a_3^2\lambda_1r^4+a_2a_4\lambda_1r^3(1+r),$ we obtain
	$$a_3^2\frac{(1+r^4)^2}{4}+a_2a_4\frac{(1+r^3)(1+r^5)}{2}-a_3^2r^4-a_2a_4r^3(1+r)=0 $$
	and then $\frac12a_2a_4= 0$ as it is the coefficient of $r^5$.
	Contradiction. \end{proof}
\begin{lem}\label{lem2.8}Let $\mu=\sum_{i=1}^6a_i\delta_{\lambda_i}$  as above such that 
	$\mu*t\mu$ has a square  root. Then either $\lambda_2^2\ne \lambda_1\lambda_{4}$ or   $\lambda_{5}^2\ne  \lambda_3\lambda_{6}$.
\end{lem} 
\begin{proof}
	Suppose that   $(\lambda_2^2,\lambda_5^2)= (\lambda_1\lambda_4, \lambda_6\lambda_3)$  (or equivalently 
	$\frac{\lambda_2}{\lambda_1}= \frac{\lambda_4}{\lambda_2}$ and $\frac{\lambda_6}{\lambda_5}= \frac{\lambda_5}{\lambda_3})$. It follows that  $\{\lambda_1\lambda_3, \lambda_6\lambda_4\}\subset{\cal U}r_\mu$. Using  Proposition \ref{regles} (1.), it follows that  $\lambda_2\lambda_3$ and  $\lambda_4\lambda_5$ are  
	not uniquely represented. Then $\lambda_2\lambda_3\in \{ \lambda_1\lambda_5,\lambda_1\lambda_6\}$ and $\lambda_4\lambda_5\in \{ \lambda_1\lambda_6,\lambda_2\lambda_6\}$ . 
	If we assume that  $\lambda_2\lambda_3=\lambda_1\lambda_6$ (or symmetrically  $\lambda_4\lambda_5=\lambda_1\lambda_6$), we will get  
	$\lambda_1\lambda_5\in{\cal U}r_\mu$, which contradicts Proposition \ref{regles} (5.).
	We derive that 
	$$    \lambda_2\lambda_3=\lambda_1\lambda_5 \mbox{ and } \lambda_4\lambda_5=\lambda_2\lambda_6 $$
	and then 
	$(r_1=r_2^2)=\frac{\lambda_2}{\lambda_1}= \frac{\lambda_4}{\lambda_2}=\frac{\lambda_5}{\lambda_3}= \frac{\lambda_6}{\lambda_5}.$

	Since $\{\lambda_1^2,\lambda_1\lambda_3\}\in{\cal U}r_\mu$ and $\{\lambda_6^2,\lambda_4\lambda_6\}\in{\cal U}r_\mu$ we get $\{\lambda_3^2\}\not\in{\cal U}r_\mu$, $\{\lambda_4^2\}\not\in{\cal U}r_\mu$.
	
	We clearly  have  $\frac{\lambda_3}{\lambda_2}= \frac{\lambda_4}{\lambda_3}= \frac{\lambda_5}{\lambda_4}=(r_2)$ and hence    $\lambda_i=r^i\lambda_1$ for $1<i<6$, and $\lambda_6=r^7\lambda_1$.
	From the description  above, we obtain:
	$$  b_1=a_1\sqrt{\lambda_1}, \; 
	b_2=\frac12 a_2\sqrt{\lambda_1}(1+r^2), \; 
	b_6=a_6\sqrt{r\lambda_1}r^3, \:         
	b_4=\frac12 a_4\sqrt{r\lambda_1}(1+r^3), $$    
	and  
	$$
	b^2_2+2b_1b_4= a_2^2 r^2\lambda_1+a_1a_4\lambda_1(1+r^4).
	$$
	This leads to the next identity, 
	$$   a_2^2 r^2+a_1a_4(1+r^4)-\frac14 a_2^2(1+r^2)^2-a_1a_4\sqrt{r}(1+r^3)=0.
	$$
	and  the coefficient  $a_1a_4$ of $\sqrt{r}$ vanishes as before. Contradiction.
\end{proof}

It follows  that only one of the 3 next cases may occur $(\lambda_2^2,\lambda_5^2)= (\lambda_1\lambda_3, \lambda_3\lambda_6)$, $(\lambda_2^2,\lambda_5^2)= (\lambda_1\lambda_4, \lambda_4\lambda_6)$ and 
$(\lambda_2^2,\lambda_5^2)= (\lambda_1\lambda_3, \lambda_4\lambda_6)$. From  the first two symmetric cases, we have

\begin{prop}\label{lem2.10}Let $\mu=\sum_{i=1}^6a_i\delta_{\lambda_i}$ be with $6$ atoms.\\  I) If $(\lambda_2^2,\lambda_5^2)= (\lambda_1\lambda_4, \lambda_4\lambda_6)$,
	then the following are equivalent
	\begin{enumerate}
		\item $\mu*t\mu$ has a square  root. 
		\item $\lambda_3^2=\lambda_1\lambda_6$ and
		$(a_2^2, a_3^2,a_5^2) =(4a_1a_4, 4a_1a_6,4a_4a_6)$.
		\item $\mu$ has a square  root. 
	\end{enumerate}
	II) If $(\lambda_2^2,\lambda_5^2)= (\lambda_1\lambda_3, \lambda_3\lambda_6)$,
	then the following are equivalent
	\begin{enumerate}
		\item $\mu*t\mu$ has a square  root. 
		\item $\lambda_4^2=\lambda_1\lambda_6$ and
		$(a_2^2, a_4^2,a_5^2) =(4a_1a_3, 4a_1a_6,4a_3a_6)$.
		\item $\mu$ has a square  root. 
	\end{enumerate}
\end{prop}
\textit{Proof.} Because of symmetry, we only show $I)$. Since the implication  $3. \Rightarrow 2.$ is well known, it remains to show  $1. \Rightarrow 2.$  and  $2. \Rightarrow 3.$\\ 
$1. \Rightarrow 2.$ From $\lambda_2^2=\lambda_1\lambda_4$, as in the proof of Lemma \ref{lem2.8} we obtain 
$   \lambda_2\lambda_3=\lambda_1\lambda_5$. Using $\lambda_5^2= \lambda_4\lambda_6, $ we get $ \lambda_3\lambda_5=\lambda_2\lambda_6$.
Thus $$ \frac{\lambda_2}{\lambda_1}=\frac{\lambda_4}{\lambda_2}= \frac{\lambda_5}{\lambda_3} =R \mbox{ and }  \frac{\lambda_3}{\lambda_2}=\frac{\lambda_5}{\lambda_4}=  \frac{\lambda_6}{\lambda_5}=r
$$
we deduce that,
$$
\begin{array}{ccccc}
	\lambda_2=R\lambda_1& \lambda_3=rR\lambda_1 & \lambda_4=R^2\lambda_1  
	& \lambda_5=rR^2\lambda_1 &  \lambda_6=(rR)^2\lambda_1
\end{array}
$$
If  $\lambda_3\lambda_6 \in {\cal NU}r_\mu $ then  clearly $\lambda_3\lambda_6 \in \{\lambda_2^4,\lambda_4\lambda_5\}$. If $\lambda_3\lambda_6 = \lambda^2_4$, we get $R=r^3$ and if $\lambda_3\lambda_6 = \lambda_4\lambda_5$, we have $R=r^2$ . Hence arises the  next 3 cases 
\begin{itemize}
	\item $R=r^2.$ We obtain  the next spectral picture of $supp(\mu*t\mu)$:
\end{itemize}   
\tiny{
$$
\begin{array}{ccccccccccccc}
	& &&&&\lambda_2\lambda_4&&  \lambda_4^2 &\lambda_4\lambda_5&&&\\
	& && && \|&& \|&\|&&&\\
	\lambda_1^2&<\lambda_1\lambda_2& <\lambda_1\lambda_3&<\lambda_1\lambda_4&<\lambda_2\lambda_3&<\lambda_1\lambda_6 &<\lambda_2\lambda_5&<\lambda_2\lambda_6&<\lambda_3\lambda_6&<\lambda_4\lambda_6&<\lambda_5\lambda_6&<\lambda_6^2 \\
	& &&\|&\|&\|&\|&\|&& \|&&\\
	& && \lambda_2^2&\lambda_1\lambda_5&\lambda_3^2 &\lambda_3\lambda_4& \lambda_3\lambda_5& &\lambda_5^2&
\end{array}
$$}\normalsize
\begin{itemize}
	\item $R=r^3.$  We get the next spectral picture of $supp(\mu*t\mu)$: \end{itemize} 
\tiny{$$ \begin{array}{cccccccccccccc}
	\lambda_1^2&<\lambda_1\lambda_2& <\lambda_1\lambda_3&<\lambda_1\lambda_4&<\lambda_2\lambda_3&<\lambda_1\lambda_6&<\lambda_2\lambda_4 &<\lambda_2\lambda_5&<\lambda_2\lambda_6&<\lambda_4^2&<\lambda_4\lambda_5&<\lambda_4\lambda_6&<\lambda_5\lambda_6&<\lambda_6^2 \\
	&&&\|&\|&\|&&\|&\|&\|&& \|&&\\
	&&&\lambda_2^2&\lambda_1\lambda_5&\lambda_3^2& &\lambda_3\lambda_4& \lambda_3\lambda_5&\lambda_3\lambda_6& &\lambda_5^2&
\end{array}
$$}\normalsize
\begin{itemize}
	\item $R\notin \{r^2, r^3\}.$  We have a "Maximal cardinal support" ( 15 atoms) and we get  the next spectral picture of $supp(\mu*t\mu)$:\end{itemize}
\tiny{$$ \hspace{-2cm}\begin{array}{ccccccccccccccc}
	\lambda_1^2&<\lambda_1\lambda_2& <\lambda_1\lambda_3&<\lambda_1\lambda_4&<\lambda_2\lambda_3&<\lambda_2\lambda_4&<\lambda_1\lambda_6 &<\lambda_2\lambda_5&<\lambda_4^2&<\lambda_2\lambda_6&<\lambda_4\lambda_5&<\lambda_3\lambda_6&<\lambda_4\lambda_6&<\lambda_5\lambda_6&<\lambda_6^2 \\
	&&&\|&\|&&\|&\|&&\|&&& \|&&\\
	&&&\lambda_2^2&\lambda_1\lambda_5&&\lambda_3^2 &\lambda_3\lambda_4&& \lambda_3\lambda_5&& &\lambda_5^2&
\end{array}
$$}\normalsize
From the associated diagrams, in the  3 cases above, We use the next common   equations to determine, $b_1, b_2, b_3,b_4, b_5$ and $  b_6,$
$$
\left\{ 
\begin{array}{lll}
	b_1^2 = a_1^2\lambda_1  & 2b_1b_2 = a_1a_2\lambda_1(1+R)\\
	2b_1b_3 = a_1a_3{\lambda_1}(1+rR)&
	b_2^2+2b_1b_4 =a_2^2 \lambda_1R+ a_1a_4{\lambda_1}(1+R^2)\\
	2b_5b_6 = a_5a_6\lambda_1rR^2(1+r) &
	b_6^2=a_6^2{\lambda_1}(rR)^2,  
\end{array}
\right.
$$
so that   $$  
\left\{ 
\begin{array}{ccc}
	b_1 = a_1\sqrt{\lambda_1},  &  b_2 = a_2\sqrt{\lambda_1}\frac{(1+R)}{2}, & b_3 = a_3\sqrt{\lambda_1}\frac{(1+rR)}{2}\\
	b_4=\frac{\sqrt{\lambda_1}}{2}\left[a_4(1+R^2)-\frac{a_2^2}{a_1}\frac{(1-R)^2}{4}\right] & b_5=a_5\sqrt{\lambda_1}\frac{R(1+r)}{2} &b_6=a_6\sqrt{\lambda_1}rR
\end{array} 
\right.
$$
Plugging in the additional 3 common equations,
$$
\left\{ 
\begin{array}{lll}
	2b_1b_5+ 2b_2b_3 &=& a_1a_5{\lambda_1}(1+rR^2)+ a_2a_3{\lambda_1}R(1+r)\\
	2b_2b_5+ 2b_3b_4 &=& a_2a_5{\lambda_1}R(1+rR)+a_3a_4{\lambda_1}R(r+R)\\    
	2b_4b_6+ b_5^2 &=& a_4a_6{\lambda_1}R^2(r^2+1)+a_5^2{\lambda_1}rR^2, \\
\end{array}
\right.
$$
we obtain the next necessary compatibility conditions:

$$ (2a_1a_5 -a_2a_3)(1-R)(1-rR) = 0,$$
$$ (2a_3a_4 -a_2a_5)(1-R^2)(1-rR) = 0,$$
$$
R{(1-r)^2}\left(4a_4-\frac{a_5^2}{a_6}\right)=0.
$$
and hence 
$$  2a_5a_1=a_2a_3, \; \; 
2a_3a_4=a_2a_5 \; \mbox{ and } a_5^2=4a_4a_6$$
It is clear that the  tree obtained relations  are equivalent to $a_2^2=4a_1a_4, \, a_3^2=4a_1a_6$ and $a_5^2=4a_4a_6$.\\
$2. \Rightarrow 3.$  Direct computations show that the 3 atomic measure  
$\sqrt{a_1}\delta_{\sqrt{\lambda_1}}+\sqrt{a_4}\delta_{R\sqrt{\lambda_1}}+\sqrt{a_6}\delta_{rR\sqrt{\lambda_1}}$
is  a square root of  $\mu$.
\begin{rem} For $p=6$, in the previous situations, the number of equations arising from the diagrams are either 12, 14 or 15. We only used the 9 common ones. As shown by Proposition \ref{lem2.10}, the remaining   compatibility   equations are automatically fulfilled.
\end{rem}

\begin{prop}\label{lem2.11} Let $\mu=\sum_{i=1}^6a_i\delta_{\lambda_i}$  be  such that $(\lambda_2^2,\lambda_5^2)= (\lambda_1\lambda_3, \lambda_4\lambda_6)$. Then ${\tilde W}_\mu$ is not subnormal.
\end{prop}
Let us start with the geometric case. Identifying the coefficients associated with $\nu*\nu$ and $\mu*t\mu$, we obtain a system of 11 nonlinear equations. We use the next six one to get expression of $b_i, i \le 6$. 
$$\begin{array}{rl}  
	b_1^2- a_1^2&=0\\
	2b_1b_2-a_1a_2(1+r)&=0\\
	(2b_1b_3+b_2^2)- (a_1a_3(1+r^2)+a_2^2r)&=0 \\
	(2b_4b_6+b_5^2)- (a_4a_6(r^3+r^5)+a_5^2r^4)&=0\\
	+2b_5b_6 -a_5a_6(r^4+r^5)&=0\\
	b_6^2 -a_6^2r^5 &=0.
\end{array}$$

We have,
$$
\begin{array}{ccc} 
	b_1= a_1, & b_2= \frac{1}{2}a_2(1+r),& b_3=  (\frac{a_3}{2}(1+r^2)-\frac{a_2^2}{8a_1}(1-r)^2)\\
	b_6= a_6r^2\sqrt{r},& b_5= \frac{1}{2}a_5(1+r)r\sqrt{r}, & b_4= ( \frac{a_4}{2}(1+r^2) -\frac{a_5^2}{8a_6}(1-r)^2)\sqrt{r}).
\end{array}
$$

Plugging in the $ 7^{th} $  equation,  $(2b_1b_4+2b_2b_3)-(a_1a_4(1+r^3)+a_2a_3(r+r^2))=0$, we derive 

$
a_1a_4[(1+r^2)\sqrt{r}-1-r^3] - \frac{a_5^2a_1}{4a_6}(1-r)^2\sqrt{r}+ \frac{1}{2}a_2a_3(1+r)(1-r)^2 -\frac{a_2^3}{8a_1}(1+r)(1-r)^2=0.
$ And because of the coefficient of $r\sqrt{r}$ that should vanish, we obtain $\frac{a_5^2a_1}{4a_6}=0$. Contradiction.

Suppose now that $supp(\mu)$ is not geometric and write  $\frac{\lambda_2}{\lambda_1} = \frac{\lambda_3}{\lambda_2}=r_1$ and 
$ \frac{\lambda_5}{\lambda_4} = \frac{\lambda_6}{\lambda_5} =r_2$.\\
Aiming to describe the spectral picture, we distinguish 3 cases\\
$ a)$ $\{\lambda_1\lambda_4,\lambda_3\lambda_6\} \subset \mathcal{NU}r_\mu$\\
$ b)$ $\lambda_1\lambda_4 \in \mathcal{U}r_\mu $ and $\lambda_3\lambda_6 \notin \mathcal{U}r_\mu$ (or by symmetry $\lambda_1\lambda_4 \notin \mathcal{U}r_\mu $ and $\lambda_3\lambda_6 \in \mathcal{U}r_\mu$.\\
$ c)$ $\{\lambda_1\lambda_4,\lambda_3\lambda_6\} \subset \mathcal{U}r_\mu$.\\
$ a)$  If  $\{\lambda_1\lambda_4,\lambda_3\lambda_6\} \subset \mathcal{NU}r_\mu$, we get the
additional  information $\lambda_1\lambda_4\in \{\lambda_2\lambda_3, \lambda_3^2\}$ and $\lambda_3\lambda_6 \in \{\lambda_4^2,\lambda_4\lambda_5\} $. Thus
\begin{itemize}
	\item $(\lambda_1\lambda_4,\lambda_3\lambda_6 )= (\lambda_2\lambda_3,\lambda_4\lambda_5 )$. This implies 
	$\frac{\lambda_2}{\lambda_1} = \frac{\lambda_3}{\lambda_2}=\frac{\lambda_4}{\lambda_3}=\frac{\lambda_5}{\lambda_4} = \frac{\lambda_6}{\lambda_5}$  and then $\mu$ has a geometric support. Contradiction.
	\item  $(\lambda_1\lambda_4,\lambda_3\lambda_6 )= (\lambda_2\lambda_3,\lambda_4^2 )$, ( or symmetrically $(\lambda_1\lambda_4,\lambda_3\lambda_6 )= (\lambda_3^2,\lambda_4\lambda_5 )$). 
	
	It will follow that $\frac{\lambda_2}{\lambda_1} =\frac{\lambda_4}{\lambda_3} =\frac{\lambda_6}{\lambda_4}$ and then $r_1=r_2^2(=r^2).$  This gives $
	\lambda_2=\lambda_1r^2 \; \lambda_3=\lambda_1r^4 \; \lambda_4=\lambda_1r^6 \;$   $\lambda_5=\lambda_1r^7\; \lambda_6=\lambda_1r^8. $
	From the above expressions it  comes that
	$\lambda_1\lambda_{5} \in  {\cal U}r_\mu$, which contradicts Proposition \ref{regles}.
	\item   $(\lambda_1\lambda_4,\lambda_3\lambda_6 )= (\lambda_3^2,\lambda_4^2)$. We get $r_1^2=\frac{\lambda_3}{\lambda_1} =\frac{\lambda_4}{\lambda_3} =\frac{\lambda_6}{\lambda_4} =r_2^2$ and then $r_1=r_2 (=r)$. We deduce that $   \lambda_2=r\lambda_1 \; \lambda_3=r^2\lambda_1 \; \lambda_4=r^4\lambda_1 \;   \lambda_5=r^5\lambda_1\; \lambda_6=r^6\lambda_1. $
	
	This  leads to the equations: 
	
	$$  
	\left\{ 
	\begin{array}{lllllll}
		b_1^2 &=& a_1^2\lambda_1, & b_6^2 &=& a_6^2{\lambda_1}r^6, \\
		2b_1b_2 &=& a_1a_2\lambda_1(1+r), &  2b_5b_6 &=& a_5a_6\lambda_1r^5(1+r),\\
		2b_2b_3 &=&a_2a_3{\lambda_1}(r+r^2), &  2b_4b_5 &=& a_4a_5{\lambda_1}r^4(1+r). 
	\end{array} 
	\right.
	$$
	
	Hence    $b_1, b_2, b_3,b_4,b_5$ and $b_6$ are given by 
	$
	b_1 = a_1\sqrt{\lambda_1}, \;   b_2 = a_2\sqrt{\lambda_1}\frac{(1+r)}{2}, \; b_3 = a_3\sqrt{\lambda_1}r, \; 
	b_4 = a_4 \sqrt{\lambda_1}r^2,  \; b_5 =a_5 \sqrt{\lambda_1}r^2\frac{(1+r)}{2} $ and $ b_6=a_6 \sqrt{\lambda_1} r^3.$
	
	Replacing in the additional  equation
	$$   2b_1b_4+ b_3^2 = a_1a_4{\lambda_1}(r^4+1)+a_3^2{\lambda_1}r^2, $$ we get :
	$$
	a_1a_4(1-r^2)^2=0.     $$ Which is impossible, and then $ {\tilde W}_\mu$  is not subnormal.
\end{itemize}  
$b)$  $\lambda_1\lambda_4 \in \mathcal{U}r_\mu $ and $\lambda_3\lambda_6 \notin \mathcal{U}r_\mu$. In this case, we have $\lambda_3\lambda_6\in \{ \lambda_4\lambda_5, \lambda_4^2\}$ and $\lambda_2\lambda_4\in{\cal NU}r_\mu$. Moreover $\lambda_2\lambda_4\in \{\lambda_1\lambda_5,\lambda_3^2,\lambda_1\lambda_6\}$.
\begin{itemize}
	\item $\lambda_3\lambda_6 = \lambda_4\lambda_5$, gives $\frac{\lambda_2}{\lambda_1} =\frac{\lambda_3}{\lambda_2}$ and $ \frac{\lambda_4}{\lambda_3}=\frac{\lambda_5}{\lambda_4}=\frac{\lambda_6}{\lambda_5}$ and if moreover  $\lambda_2\lambda_4\in \{\lambda_1\lambda_5,\lambda_3^2\}$, we obtain  $\frac{\lambda_2}{\lambda_1}=\frac{\lambda_3}{\lambda_2} =\frac{\lambda_4}{\lambda_3} =\frac{\lambda_5}{\lambda_4}=\frac{\lambda_6}{\lambda_5}$ that contradicts $\lambda_1\lambda_4 \in \mathcal{U}r_\mu.$ It follows that $\lambda_2\lambda_4=\lambda_1\lambda_6$ and  then $r_1=r_2^2.$ This gives  $
	\lambda_2=r^2\lambda_1 \; \lambda_3=r^4\lambda_1 \; \lambda_4=r^5\lambda_1 \;   \lambda_5=r^6\lambda_1\; \lambda_6=r^7\lambda_1. $ It results that
	$$
	\left\{ 
	\begin{array}{l}
		b_1^2 = a_1^2\lambda_1,  \quad \quad \quad\quad\quad \quad \quad b_6^2=a_6^2{\lambda_1}r^7\\
		2b_1b_4= a_1a_4\lambda_1(1+r^5),  \quad 2b_5b_6= a_5a_6r^6\lambda_1(1+r).\\
	\end{array}
	\right.
	$$
	Then
	$$
	\begin{array}{cccc}
		b_1 = a_1\sqrt{\lambda_1},& b_4 =a_4 \sqrt{\lambda_1} \frac{(1+r^5)}{2}, &    b_5 = a_5\sqrt{\lambda_1}r^{\frac52}\frac{(1+r)}{2}, & b_6=a_6 \sqrt{\lambda_1} r^{\frac72}
	\end{array} 
	$$
	Replacing in the additional equation:
	$2b_4b_6 + b_5^2=a_4a_6{\lambda_1}r^5(1+r^2)+a_5^2\lambda_1r^6 $, we find
	$$
	a_5^2 r^{\frac32}\frac{(1-r)^2}{4}+a_4a_6(r^{\frac32}-1)(r^{\frac72}-1)=0.
	$$
	Which is again impossible and hence  $ {\tilde W}_\mu$ is not subnormal.
	\item $\lambda_3\lambda_6 = \lambda_4^2.$ As before  $\lambda_2\lambda_4\ne \lambda_3^2$. Here we have two subcases:\\
	$i)$  $\lambda_2\lambda_4 = \lambda_1\lambda_5$.  Which gives 
	$ \lambda_2=r\lambda_1 \; \lambda_3=r^2\lambda_1 \; \lambda_4=r^4\lambda_1 \;   \lambda_5=r^5\lambda_1\; \lambda_6=r^6\lambda_1. $ Thus $\lambda_1\lambda_4=\lambda_3^2$, contradiction.\\
	$ii)$   $\lambda_2\lambda_4= \lambda_1\lambda_6$, gives   $\frac{\lambda_6}{\lambda_4}=\frac{\lambda_2}{\lambda_1}$ and we have
	$\frac{\lambda_4}{\lambda_3} =\frac{\lambda_6}{\lambda_4}$, we obtain $\frac{\lambda_2}{\lambda_1}=\frac{\lambda_4}{\lambda_3}$  that contradicts $\lambda_1\lambda_4 \in \mathcal{U}r_\mu $.
\end{itemize}
$ c)$ $\{\lambda_1\lambda_4,\lambda_3\lambda_6\} \subset \mathcal{U}r_\mu$. We deduce that
$\{\lambda_2\lambda_4,\lambda_3\lambda_5\} \subset \mathcal{NU}r_\mu$. 
Since $\lambda_2\lambda_4=\lambda_3^2$ will imply $\lambda_1\lambda_4=\lambda_2\lambda_3$, we have 
$\lambda_2\lambda_4\ne\lambda_3^2$ and similarly $\lambda_3\lambda_5\ne\lambda_4^2$. Thus $\lambda_2\lambda_4\in \{\lambda_1\lambda_5, \lambda_1\lambda_6\}$ and $\lambda_3\lambda_6\in \{\lambda_2\lambda_6, \lambda_1\lambda_6\}$.\\
$i)$  If we suppose now that $\lambda_2\lambda_4=\lambda_1\lambda_5$ and $\lambda_3\lambda_5= \lambda_2\lambda_6$, we will have $r_1=r_2$ and hence 
$\frac{\lambda_2}{\lambda_1} =\frac{\lambda_3}{\lambda_2}=\frac{\lambda_5}{\lambda_4}=\frac{\lambda_6}{\lambda_5}(=r)$. We deduce that 
$\lambda_2=r\lambda_1,\lambda_3=r^2\lambda_1,  \lambda_5=r\lambda_4$ and $\lambda_6=r^2\lambda_4$. 
If we  show that $\lambda_4^2 \in{\cal U}r_\mu$, we will get a contradiction since $\{\lambda_1^2, \lambda_1\lambda_4\} \subset  {\cal U}r_\mu$.\\
Suppose $\lambda_4^2=\lambda_i\lambda_j$ for $i\le 3$ and $j\ge 5$. We will have $\lambda_4^2=r^{i-1}r^{j-4}\lambda_1\lambda_4$ and thus 
$\lambda_4=r^{i+j-5}\lambda_1$. Now since $i+j-5 \le 3$ and $\lambda_4 \notin\{\lambda_2, \lambda_3\}$, it follows that  $\lambda_4= \lambda_1r^3$ and then 
$ \lambda_1\lambda_4=\lambda_2\lambda_2$. Impossible\\
$ii)$ $(\lambda_2\lambda_4, \lambda_3\lambda_5) =(\lambda_1\lambda_6, \lambda_1\lambda_6), $ that gives $\lambda_2\lambda_4= \lambda_3\lambda_5$ and hence $\lambda_1\lambda_4= \lambda_2\lambda_5$, impossible.\\
$iii)$ $(\lambda_2\lambda_4, \lambda_3\lambda_5) =(\lambda_1\lambda_5, \lambda_1\lambda_6)$. We will obtain $r_1=r_2$ and $r_2=r_1^2$. Contradiction again. \\
$(iv)$ $(\lambda_2\lambda_4, \lambda_3\lambda_5) =(\lambda_1\lambda_6\lambda_2\lambda_6)$. Is the symmetric case of $(iii)$, which run in the same manner.
\section{Concluding remarks}

As shown in the case $p=6$, the number of atoms of $\mu*t\mu$ in the case where $\mu$ admits a square root may be $12, 14$ or $15$. It is intriguing to determine the set of integers $q$ such that there exist $\mu$ for which $\mu*t\mu$ admits a square root and such that $card(\mu*t\mu)= q$.\\

It is not difficult to see that the method used in this paper provides an algorithm to solve the problem for arbitrary number of atoms. If is not too long to treat this problem for small number of atoms (say for $p\le 10$), it is more complicated to run in all computations for big numbers of atoms.\\

Recall that  a weighted shift is said to be  recursively generated   if its moment sequence $(\gamma_n)_{n\ge 0}$satisfies 
$$\gamma_{n+r}= a_{r-1}\gamma_{n+r-1}+ a_{r-2}\gamma_{n+r-2}+\cdots +a_1\gamma_{n+1}+a_{0}\gamma_{n},$$ 
where $a_1, \cdots, a_{r-1}$ are given numbers and $r\ge 1$ is an integer.
It is known that a subnormal weighted shift is recursively generated if and only if the associated Berger measure in finitely atomic. 
In connection with the investigations of this paper, we mention the next partial version of a question  stated in \cite{exn} as question 4.2. 

\textit{"We do not know of a recursively generated subnormal weighted shift (with
	the trivial exceptions of scalar multiples of the unweighted shift)  for which any of the Aluthge transform,
	or root shifts are subnormal."}

The approach used  of this paper provides a promising way to solve this problem. For recursively generated shift of order $ r\le 6$, a complete description is given.
	
\end{document}